\documentclass[reqno]{amsart}
\usepackage{amsmath}
\usepackage{amsfonts}
\usepackage{amssymb,enumerate}
\usepackage{amsthm}
\usepackage[all]{xy}
\usepackage{rotating}
\usepackage{hyperref}
\theoremstyle{plain}
\newtheorem{lem}{Lemma}[section]
\newtheorem{cor}[lem]{Corollary}
\newtheorem{prop}[lem]{Proposition}

\newtheorem{intthm}{Theorem}

\theoremstyle{definition}
\newtheorem{defn}[lem]{Definition}
\newtheorem{ex}[lem]{Example}

\newtheorem{disc}[lem]{Remark}

\newtheorem{construction}[lem]{Construction}

\newtheorem{fact}[lem]{Fact}
\newtheorem{para}[lem]{}

\newtheorem*{convention*}{Convention}

\newcommand{\pd}{\operatorname{pd}}

\newcommand{\id}{\operatorname{id}}	
\newcommand{\fd}{\operatorname{fd}}

\newcommand{\cidim}{\mathrm{CI}\text{-}\!\dim}

\newcommand{\rank}{\operatorname{rank}}	

\newcommand{\edim}{\operatorname{edim}}

\newcommand{\DD}{\operatorname{D}}
\newcommand{\ann}{\operatorname{Ann}}

\newcommand{\cx}{\operatorname{cx}}

\newcommand{\ch}{\operatorname{char}}

\newcommand{\HH}{\operatorname{H}}
\newcommand{\Hom}{\operatorname{Hom}}	

\newcommand{\spec}{\operatorname{Spec}}

\newcommand{\aq}{\operatorname{AQ-dim}}

\newcommand{\Ker}{\operatorname{Ker}}

\newcommand{\ideal}[1]{\mathfrak{#1}}
\newcommand{\m}{\ideal{m}}
\newcommand{\n}{\ideal{n}}
\newcommand{\p}{\ideal{p}}
\newcommand{\q}{\ideal{q}}
\newcommand{\fm}{\ideal{m}}
\newcommand{\fn}{\ideal{n}}
\newcommand{\fp}{\ideal{p}}

\newcommand{\fr}{\ideal{r}}

\newcommand{\bbn}{\mathbb{N}}

\newcommand{\bbr}{\mathbb{R}}

\newcommand{\from}{\leftarrow}
\newcommand{\xra}{\xrightarrow}
\newcommand{\xla}{\xleftarrow}

\newcommand{\vf}{\varphi}

\renewcommand{\geq}{\geqslant}
\renewcommand{\leq}{\leqslant}
\renewcommand{\ker}{\Ker}

\newcommand{\Ext}[4][R]{\operatorname{Ext}_{#1}^{#2}(#3,#4)}

\renewcommand{\Hom}[3][R]{\operatorname{Hom}_{#1}(#2,#3)}	
\newcommand{\Tor}[4][R]{\operatorname{Tor}^{#1}_{#2}(#3,#4)}

\newcommand{\cfd}{\operatorname{CI-fd}}

\newcommand{\cdim}{\operatorname{CI\text{-}dim}}

\numberwithin{equation}{lem}

\begin{document}

\bibliographystyle{amsplain}

\title[Andr\'{E}-Quillen homology and CI flat dimension]{Andr\'{E}-Quillen homology and complete intersection flat dimension}

\author[T. Sharif]{Tirdad Sharif}

\address{School of Mathematics,
Institute for research in fundamental sciences (IPM),
P. O. Box 19395-5746, Tehran, Iran.}
\email{sharif@ipm.ir}

\thanks{Tirdad Sharif was supported by a grant from IPM (No. 83130311).}

%\date{\today}


\begin{abstract}
We study the relation between the vanishing
of Andr\'{e}-Quillen homology and complete intersection flat dimension and we extend
some of the existing results in the literature.
\end{abstract}

\maketitle

\section{Introduction}
\label{introduction}

\begin{convention*}\label{notation130604a}
In this paper, rings are commutative, noetherian, and local
with identity.
In the entire paper, $(R, \fm, k)$, $(S,\n, \ell)$, and $(T,\fr, v)$ are local rings
where $S$ and $T$ are complete. When it is convenient, $\varphi\colon R\to S$ and $\sigma\colon S\to T$ will denote local ring
homomorphisms.
\end{convention*}

Over forty years ago, Andr\'{e}~\cite{An1, An2} and Quillen~\cite{Q1,Q2}
introduced a homology theory for commutative algebras that
is known as Andr\'{e}-Quillen (or cotangent)
homology. The $n$-th Andr\'{e}-Quillen homology of the
$R$-algebra $S$ with coefficients in an $S$-module $N$, denoted $\DD_n(S\mid R,N)$,
is the $n$-th
homology module of $\mathcal{L_{\varphi}}\otimes_S N$, where
$\mathcal{L_{\varphi}}$ is the cotangent complex of $\varphi$,
uniquely defined in the derived category of the category of $S$-modules.
The vanishing of Andr\'{e}-Quillen homology characterizes
important classes of rings and ring homomorphisms; see;
\cite{An1,An3,Av,AHS,AI2,AI1,Q1,Q2}.

To study the relation between the vanishing of Andr\'{e}-Quillen homology
and the structure of a local ring homomorphism $\vf$, Avramov and Iyengar~\cite{AI2}
introduced the notion of Andr\'{e}-Quillen dimension of the $R$-algebra $S$, denoted $\aq_RS$, that is
$$
\aq_RS:=\sup\{n\in \bbn\mid \DD_n(S\mid R,-)\neq 0\};
$$
in particular, $\aq_{R} S=-\infty$ if and only if
$\DD_n(S\mid R,-)=0$ for all integers $n$.

In this paper, we investigate the relation between the vanishing
of Andr\'{e}-Quillen homology and complete intersection flat dimension.
As the first application of this study, we prove the following theorem which is one of our main results in this paper; see~\ref{para130611a} for the proof.

\begin{intthm}\label{thm130316b}
Assume that $\sigma$ is surjective, $\sigma\varphi$ is complete intersection at $\fr$, and $\aq_{S} T <\infty$.
Then $\cidim_S T<\infty$.
\end{intthm}

It is worth mentioning that the surjectivity condition in this theorem can be removed if one can give an affirmative answer to a question posed by Foxby; see Remark~\ref{disc131027a}.

\vspace{0.05in}

The second main result of this paper is the following theorem, which extends a result of Soto~\cite[Proposition 12]{S} and of Avramov,
Henriques, and \c{S}ega~\cite[(2.5)(3)]{AHS}; see~\ref{para130611b} for the proof.
Here, $\HH_1(K.(\vf)$), denotes the first Koszul homology of $\vf$; see Definition~\ref{defn130610a}

\begin{intthm}\label{thm130316c}
Assume that $\vf$ is essentially of finite type and $\cfd_{R} S<\infty$. Then $\aq_{R} S\leq 2$ if and only if
$\HH_1(K.(\vf))$ is a free $S$-module.
\end{intthm}

\section{Some Background Material}
\label{background}

This section lists important foundational material for use in this paper.
We use the notations $\mu_R(M)$ and $\edim(R)$ for
the minimal number of generators of a finitely generated $R$-module $M$ and the embedding dimension of the ring $R$, respectively.

The objects defined in the next definition are from the work of
Avramov, Foxby, and B. Herzog~\cite{AFH}.

\begin{para}
A \emph{regular  factorization of $\vf$}
is a diagram of local ring homomorphisms
$R\xra{\dot{\varphi}} R'\xra{\varphi'} S$
such that $\varphi=\varphi'\dot{\varphi}$, the map $\dot{\varphi}$ is
flat, the closed fibre $R'/\fm R'$ is regular, and $\varphi'$ is surjective.
A \emph{Cohen factorization of $\vf$} is a regular factorization
$R\xra{\dot{\varphi}} R'\xra{\varphi'} S$ of $\vf$ such that $R'$ is complete.
Since we assumed that $S$ is complete, existence of a Cohen factorization of $\vf$ is
guaranteed by~\cite[Theorem (1.1)]{AFH}.
\end{para}

\begin{para}
The local ring homomorphism $\varphi$ is called \emph{complete intersection
at $\n$} if in a given Cohen factorization $R\to R'\xra{\varphi'} S$ of
$\varphi$ the ideal $\ker \varphi'$ is generated by an $R'$-regular sequence.
The complete intersection property for $\varphi$ is independent of the choice of Cohen factorization
by~\cite[Remark (3.3)]{Av}.
Also $\varphi$ is called \emph{locally complete
intersection} if for each $\q\in \spec(S)$, the local homomorphism $\varphi_\q: R_{\p}\rightarrow S_{\q}$, in which $\p=\q\cap R$, is complete intersection at $\q S_{\q}$.
\end{para}

We proceed with introducing the complete intersection flat dimension as defined by Sather-Wagstaff~\cite{SSW1}.

\begin{para}\label{fact130625a}
Fix a an arbitrary $R$-module $M$.
Define
\begin{align*}
\cfd_R M
:=\inf\left\{\fd_Q(R'\otimes_R M)-\pd_Q(R')\left| \text{
\begin{tabular}{@{}c@{}}
$R\to R'\from Q$ is a \\ quasi-deformation
\end{tabular}
}\!\!\!\right. \right\}.
\end{align*}
Recall that a diagram $R\to R'\from Q$ of local ring homomorphisms is a \emph{quasi-deformation of $R$} if $R\to R'$ is
flat and $R'\xla{}Q$ is surjective with kernel
generated by a $Q$-regular sequence. It is clear that if $M$ is finitely generated then $\cidim_R M=\cfd_R M$.
\end{para}

\begin{para}
Here, we recall some notions from~\cite{AI2, L}.

Let $\overline{\vf}\colon k\to \ell$ be the induced map between the
residue fields. Then $\varphi$ is called \emph{almost
small} if the kernel of the homomorphism
$$
\Tor[\varphi]{*}{\overline{\varphi}}\ell\colon
\Tor[R]{*}k\ell\to \Tor[S]{*}\ell\ell
$$
of graded algebras is
generated by elements of degree 1.
Assuming $k=\ell$, the homomorphism $\vf$
is called \emph{large} if $\Tor[\varphi]{*}{\overline{\varphi}}k$ is surjective.
\end{para}

The next definition originates with work of Gulliksen and Levin~\cite{GL}; see~\cite[\S 6]{A} for details.
It involves a different type of algebra extensions introduced by Tate~\cite{T};
see~\cite{A, GL} for details.

\begin{para}
Let $A$ be a DG
$R$-algebra, and assume that $\varphi$ is surjective. Let $A\langle Y\rangle$ denote a DG $R$-algebra
obtained from $A$ by adjunctions of sets of exterior
variables of positive odd degrees and of divided power
variables of positive even degrees. Then a factorization
$R\longrightarrow R\langle
Y\rangle\stackrel{\simeq}\longrightarrow S$ of
$\varphi$ is called an
\emph{acyclic closure of $\varphi$} if $\partial(Y_1)$ minimally
generates $\ker \varphi$, and \{cls$(\partial(y))\mid y\in Y_{n+1}$\} is
a minimal generating set of $\HH_n(R\langle Y_{\leq n}\rangle))$ for
all $n\geq 1$.
\end{para}

\begin{fact}\label{fact1}
If $J=(x)=(y)$ is an ideal of $R$, such that $x$ and $y$ are two minimal set of generators for $J$, then the Koszul complexes associated to $(x)$ and $(y)$ are isomorphic, see \cite{BH}, and hence homology isomorphism. Thus we define the Koszul complex, $K.^{R}(J)$, to be the Koszul complex associated to an arbitrary minimal set of generators for ideal $J$.
\end{fact}

\begin{fact}\label{fact2}
Let $R\xra{\dot{\varphi_j}} R_j\xra{\varphi_j} S$, with $J_j=\Ker\varphi_j$, for $j=1,2$, be two arbitrary Cohen factorization of $\varphi$.
From~\cite[(1.2)]{AFH} it follows that they have a common deformation $R\to R'\xra{\varphi'} S$, with $J'=\Ker\varphi'$ such that $R_j=R'/I_j'$ and $J_j=J'/I_j'$, for $j=1,2$, where $I_j'$ is an ideal generated by an $R'$-regular sequence that extends to a minimal set of generators for $J'$. Let $\underline{\textbf{x}}_j$ be a minimal set of generators for $I_j'$, which can be extended to a minimal set of generators $\underline{\textbf{x}}_j, \underline{\textbf{y}}_j$ for $J'$. Then the residues $\underline{\textbf{y}}_j'$ of  $\underline{\textbf{y}}_j$ in $R_j$ form a minimal set of generators for $J_j$, by Nakayama's lemma. Now from the above fact we have
$$
K.^{R'}(J')\cong K.^{R'}(\underline{\textbf{x}}_j)\otimes_{R'} K.^{R'}(\underline{\textbf{y}}_j)\simeq R_j\otimes_{R'}K.^{R'}(\underline{\textbf{y}}_j)\cong K.^{R_j}(\underline{\textbf{y}}_j')\cong K.^{R_j}(J_j).
$$
\noindent Thus $K.^{R'}(J')\simeq K.^{R_j}(J_j)$, as $R'$-complexes, for $j=1,2$. Since $J'\HH(K.^{R'}(J'))=0$, we get $\HH(K.^{R'}(J'))\cong \HH(K.^{R_j}(J_j))$, as $S$-modules, for $j=1,2$.
\end{fact}

\vspace{0.05in}

\noindent By the above fact, we can give the following definition.

\begin{defn}\label{defn130610a}
Let $R\to R'\xra{\varphi'} S$ be a Cohen factorizations of $\varphi$, with $J'=\Ker\varphi'$. We define the \emph{nth Koszul homology} of $\varphi$, denoted by $\HH_n(K.(\varphi))$, to be $\HH_n(K.^{R'}(J'))$.
\end{defn}

\noindent Indeed, we can consider $\HH(K.(\varphi))$, as an invariant of the local homomorphism $\varphi$.

\begin{lem}\label{para131027a}
Let $R\to R'\xra{\varphi'} S$ be a Cohen factorization of $\varphi$. If $E$ is the Koszul complex associated to an arbitrary set of generators ${\it \underline{\textbf{v}}}$ for $J':=\Ker \vf'$. Then $\HH_1(E)$ is a free $S$-module if and only if $\HH_1(K.(\varphi))$ is a free $S$-module.
\end{lem}
\begin{proof}
Let $\underline{\textbf{v}}=x_1,...,x_p,x_{p+1},...,x_n$, in which $\underline{\textbf{w}}=x_1,...,x_p$ is a minimal set of generators for $J'$. Then from~\cite[(1.6.21)]{BH} and Fact 2.6, we get
$$
\HH_1(E)\cong \HH_1(K.^{R'}(\underline{\textbf{w}}))\oplus S^{\oplus(n-p)}\cong \HH_1(K.^{R'}(J'))\oplus S^{\oplus(n-p)}.
$$
Now from the above fact, assertion holds.
\end{proof}

\begin{para}
Let $M$ be a finitely generated $R$-module. The \emph{complexity of $M$
over $R$}, denoted by $\cx_R M$,
is the number
$$
\cx_R M:=\inf\{d\in \bbn\mid \text{there exists}\ \gamma\in \bbr\ \text{such that}\
 \beta^R_n(M)\leq \gamma n^{d-1}\ \text{for all $n\gg 0$}\}
$$
where $\beta^R_n(M):=\rank_k(\Tor[R]nMk)$ is the $n$-th \emph{Betti number} of $M$.
\end{para}

\section{Main results}

\label{sec130315a}
This section contains the proofs of Theorems~\ref{thm130316b} and~\ref{thm130316c}.
We begin by a construction used by Iyengar and
Sather-Wagstaff~\cite{IW}, essentially given by Avramov, Foxby, and B. Herzog~\cite{AFH}. We use this construction in the proof of Theorem~\ref{thm130316b} and in Remark~\ref{disc131027a}.

\begin{construction}\label{fact130607c}
Let $R\xra{\dot{\vf}}R'\xra{\vf'} S$ and $R'\xra{\dot{\rho}}
R''\xra{\rho'} T$
be Cohen factorizations of $\vf$ and $\sigma\vf'$,
respectively. Then $\rho'$ factors through the tensor
product $S'=R''\otimes_{R'}S$ and gives the following commutative
diagram of local ring homomorphisms
\[
\xymatrixrowsep{2.5pc}
\xymatrixcolsep{2.5pc}
\xymatrix{
&& R''
\ar@{->}[dr]^{\vf''}
\ar@/^2.5pc/[ddrr]^{\rho'=\sigma'\vf''}
\\
& R'
\ar@{->}[dr]^{\vf'}
\ar@{->}[ur]^{\dot\rho}
&& S'
\ar@{->}[dr]^{\sigma'}
\\
R
\ar@{->}[rr]^{\vf}
\ar@/^2.5pc/[uurr]^{\dot\rho \dot\vf}
\ar@{->}[ur]^{\dot\vf}
&& S
\ar@{->}[rr]^{\sigma}
\ar@{->}[ur]^{\dot\sigma}
&& T
}
\]
in which $\dot{\sigma}$ and $\vf''$ are the natural maps to
the tensor product and the diagrams
$S\to S'\to T$,
$R\to R''\to T$, and
$R'\to R''\to S'$ are Cohen factorizations.
\end{construction}

Here is the proof of Theorem~\ref{thm130316b}.

\begin{para}[Proof of Theorem~\ref{thm130316b}]\label{para130611a}
We use the notation from Construction~\ref{fact130607c}. Since $\sigma$ is surjective, we get the following diagram
\begin{equation}\label{eq130607a}
R'\xra{\vf'} S\xra{\sigma}T
\end{equation}
of surjective local ring homomorphisms such that $R\xra{\dot \vf} R'\xra{\sigma \vf'} T$ is a Cohen factorization.
Since $\sigma\varphi$ is complete intersection at $\fr$, from~\cite[(1.7), (1.8)]{Av} we get that
$\sigma\vf'$ is also complete intersection at
$\fr$. Hence, by~\cite[Corollary (4.9)]{AI2} we conclude that
$\vf'$ is an almost small ring homomorphism. Since $\DD_n(T\mid S,v)=0$
for all $n\gg 0$, from the Jacobi-Zariski exact sequence arising
from diagram~\eqref{eq130607a}, we find that $\DD_n(S\mid R',v)=0$ for all $n\gg 0$.
Since $\vf'$ is almost small, it follows from~\cite[(6.4)]{AI2}
that $\vf'$ is complete intersection at $\fr$. Thus, $\Ker \vf' $ is generated by an
$R'$-regular sequence, and the diagram $S\xra{\id} S\xla{\vf'} R'$ of local ring homomorphisms,
is a quasi-deformation. Hence, we conclude that $\pd_{R'}(T\otimes_{S} S)=\pd_{R'} T<\infty$. Therefore,
$\cidim_{S} T<\infty$. \qed
\end{para}

The following lemma is an extension of~\cite[Lemma 1(2)]{R1}. This lemma is well-known; see ~\cite{Avramov}. For convenience of the reader we give a proof based on the Cohen factorization.

\begin{lem}\label{lem130316a}
If $\fd_{R}S <\infty$, then the homomorphism
$\DD_2(S\mid R,\ell)\xra{\theta} \DD_2(\ell\mid R,\ell)$ is
trivial.
\end{lem}

\begin{proof}
Let $R\xra{\dot{\vf}} R'\xra{\vf'} S$ be a Cohen factorization of $\vf$, and consider the following
commutative diagram of local ring homomorphisms
$$\xymatrix{
R\ar[r]^{\dot{\vf}} \ar[d]_{\vf'} & R'\ar[r]^{\vf'} \ar[d]^{\dot{\vf}} & S\ar[d]^{\pi} \\
S\ar[r]^{\id}\ar[d]^{\pi} & S\ar[r]^{\pi}\ar[d]^{\pi} & \ell \\
\ell \ar[r]^{\id} & \ell}
$$
where $\pi$ is the natural surjection. This diagram induces a commutative diagram
\begin{equation}\label{eq130609b}
\begin{split}
\xymatrix{
\DD_2(S\mid R,\ell)\ar[r]^{\theta}\ar[d]& \DD_2(\ell\mid R,\ell)\ar[d]^{\eta}\\
\DD_2(S\mid R',\ell)\ar[r]^{\alpha}& \DD_2(\ell\mid R',\ell)
}
\end{split}\end{equation}
of $\ell$-vector spaces.

Setting $\overline{R'}:=R'/\frak mR'$, from the Jacobi-Zariski exact sequence of $k\xra{}\overline{R'}\xra{}\ell$ we get the following exact sequence:
\begin{equation}\label{eq130609a}
\xymatrix{
\DD_{n+1}(\ell \mid \overline{R'},\ell)\xra{} \DD_n(\overline{R'}\mid k,\ell)\xra{} \DD_n(\ell \mid k,\ell)
}
\end{equation}
From ~\cite[(1.6)(2)-(3)]{Av} we get $\DD_n(\ell \mid k,\ell)=0$ for all $n\geq 2$, and $\DD_{n+1}(\ell \mid \overline{R'},\ell)=0$
for all $n\geq 2$. Hence, by \eqref{eq130609a} we conclude that $\DD_n(\overline{R'}\mid k,\ell)=0$ for all $n\geq 2$.
On the other hand, by
flat base change we have $\DD_n(\overline{R'}\mid k,\ell)\cong\DD_n(R'\mid R,\ell)$ for
all integers n; see ~\cite[4.54]{An2}. Therefore, $\DD_n(R'\mid R,\ell)=0$ for
all $n\geq 2$. Now,
the Jacobi-Zariski exact sequence arising from $R\xra{} R' \xra{}\ell$ implies that $\eta$ (from
diagram \eqref{eq130609b}) is injective.

From ~\cite[(3.2) Lemma]{AFH} and our assumption, we have $\pd_{R'} S < \infty$. Hence by ~\cite[ Lemma 1(2)]{R1} we get that $\alpha$ (from diagram \eqref{eq130609b}) is trivial. Therefore, commutativity of diagram \eqref{eq130609b} and the fact that $\eta$ is injective imply that $\theta$ is trivial.
\end{proof}

\begin{lem}\label{prop130720a}
If $\aq_{R} S\leq 2$, then
$\HH_1(K.(\vf))$ is a free $S$-module.
\end{lem}

\begin{proof}
Let $R\xra{\dot{\vf}} R'\xra{\vf'} S$ be a Cohen factorization of $\vf$, and let $J':=\Ker \varphi'$.
By assumption we get $\DD_n(S\mid R,\ell)=0$ for all $n\geq 3$. Hence, $\DD_n(S\mid R',\ell)=0$ for all $n\geq 3$ by~\cite[Lemma (1.7)]{Av}. This implies that $\DD_n(S\mid R',-)=0$ for all $n\geq 3$. Now~\cite[Corollary 3]{BMR} implies that $\HH_1(E')$, in which $E'$ is the Koszul complex associated to an arbitrary set of generators for $J'$, is a free $S$-module. The assertion now follows from Lemma~\ref{para131027a}.
\end{proof}

Here is the proof of Theorem~\ref{thm130316c}.

\begin{para}[Proof of Theorem~\ref{thm130316c}]\label{para130611b}
Since $\varphi$ is essentially of finite type, it admits a regular
factorization $R\to R[X]_{\frak M}\to S$, in which
$R[X]:=R[x_1,x_2,...,x_n]$, for some $n\geq 1$, and
$\frak M$ is
a prime ideal of $R[X]$ lying over $\m$.
Since $\cfd_{R} S<\infty$, there is a quasi-deformation
$R\to R''\xla{\eta} Q$ of $R$ such
that $\fd_{Q}(S\otimes_{R} R'')<\infty$. Set $R':=R[X]_{\frak M}$,
$P:=R'\otimes_{R} R''$, $U:=S\otimes_{R} R''$, and consider the diagram
$Q\xra{\alpha} P
\xra{\beta} U$ of natural ring
homomorphisms, in which $\beta$ is surjective.
Let $\frak u\in \spec(U)$, and set
$\p:=\beta^{-1}(\frak u)$ and
$\q:=\alpha^{-1}(\p)$. Then we have the diagram
\begin{equation}\label{eq130610a}
Q_{\q}\xra{\alpha_{\q}}
P_{\p}\xra{\beta_{\p}}
U_{\frak u}
\end{equation}
of local ring homomorphisms
that gives us the
commutative diagram
$$\xymatrix{
Q_{\q}\ar[r]^{\alpha_{\q}} \ar[d]_{\beta_{\p}\alpha_{\q}} & P_{\p}\ar[r]^{\beta_{\p}} \ar[d]^{\beta_{\p}} & U_{\frak u}\ar[d] \\
U_{\frak u}\ar[r]^{\id}\ar[d] & U_{\frak u}\ar[r]\ar[d] & k(\frak u)\\
k(\frak u)\ar[r]^{\id} & k(\frak u)}
$$
of local ring homomorphisms.
This diagram induces the following commutative
diagram of $k(\frak u)$-vector spaces
\begin{equation}\label{eq130610b}
\begin{split}
\xymatrix{
\DD_3(k(\frak u)\mid U_{\frak u},k(\frak u))\ar[r]^{f} & \DD_2(U_{\frak u}\mid P_{\fp},k(\frak u))\ar[r] & \DD_2(k(\frak u)\mid P_{\fp},k(\frak u))\\
\DD_3(k(\frak u)\mid U_{\frak u},k(\frak u))\ar[r]^{g}\ar@{=}[u]& \DD_2(U_{\frak u}\mid Q_{\q},k(\frak u))\ar[r]^{\theta}\ar[u]^{h} & \DD_2(k(\frak u)\mid Q_{\q},k(\frak u))\\
& \DD_2(P_{\fp}\mid Q_{\q},k(\frak u))\ar[u]^{\lambda}}
\end{split}\end{equation}

Since $\fd_{Q_{\q}} U_{\frak u}<\infty$, it follows from Lemma~\ref{lem130316a} that
$\theta$ is trivial and hence $g$ is surjective, and since $\HH_1(K.(\vf))$ is a free $S$-module, we can see that
$\HH_1(E)$ is a free $U_{\frak u}$-module, where $E$ is the Koszul
complex associated to the set of generators for the ideal $\Ker \beta_{\p}$ of $P_{\fp}$ arising from a set of generators for $\Ker \beta$.
Now from~\cite[(1.6.21)]{BH} it follows that $\HH_1(K.^{P_{\fp}}(\Ker(\beta_{\p})))$ is a free $U_{\frak u}$-module. Thus,
from~\cite[Theorem 1]{R2} we conclude that $f$ is trivial. Since $\DD_n(R[X]\mid R,-)=0$ for all $n\geq 1$, we get
$\DD_n(R'\mid R,-)=0$ for all $n\geq 1$. The
faithful flatness of $R''$ over $R$ also yields $\DD_n(P\mid R'',-)=0$
for all $n\geq 1$.
Hence, setting $\p'':=\p\cap R''$, we get $\DD_n(P_{\fp}\mid R''_{\p''}, k(\p))=0$ for all $n\geq 1$.

Consider the diagram $Q_{\q}\xra{\eta_{\q}} R''_{\p''}\xra{}
P_{\fp}$ of local ring homomorphisms. Since $\Ker \eta_{\q}$ is generated by a $Q_{\q}$-regular
sequence, from the Jacobi-Zariski exact sequence we get
$\DD_n(P_{\fp}\mid Q_{\q},k(\p))=0$ for all $n\geq 2$. Thus,
$h$ in diagram~\eqref{eq130610b} is injective. Commutativity of
diagram~\eqref{eq130610b} implies that $\DD_2(U_{\frak u}\mid Q_{\q},k(\frak u))=0$. Hence,
by~\cite[Proposition (1.8)]{Av} we get $\DD_n(U_{\frak u}\mid Q_{\q},k(\frak u))=0$
for all $n\geq 2$. Now, from the Jacobi-Zariski exact sequence of
diagram~\eqref{eq130610a} we conclude that $\DD_n(U_{\frak u}\mid P_{\p},k(\frak u))=0$
for all $n\geq 3$.
Since $\frak u$ was an arbitrary prime ideal of $U$,
we conclude that $\DD_n(U\mid P,-)=0$ for all
$n\geq 3$; see~\cite{An2}. Now, faithful flatness of $R''$ over $R$ implies that
$\DD_n(S\mid R',-)=0$ for all $n\geq 3$. Thus, from~\cite[Lemma (1.7)]{Av} we have
$\DD_n(S\mid R,\ell)=0$ for all $n\geq 3$. Since
$\varphi$ is essentially of finite type, from~\cite[Lemma 8.7]{I} we
conclude that $\aq_{R} S\leq 2$.

The converse follows immediately from Lemma~\ref{prop130720a}. \qed
\end{para}

Foxby asked the following question in his personal communications.

\vspace{0.05in}
\noindent{\bf Foxby's Question}. Let $R\xra{\dot{\varphi}} R'\xra{\varphi'} S$ be a Cohen factorization of
$\varphi$, and let $M$ be a finitely generated $S$-module. Do the following inequalities hold?
$$
\cfd_{R} M\leq \cidim_{R'}M\leq \cfd_{R}N+\edim(R'/\fm R')
$$
%\vspace{0.05in}

In the next remark, we show that an affirmative answer to this question helps us to generalize Theorem~\ref{thm130316b}. In fact, validity of the left inequality of Foxby's Question enables us to remove the hypothesis ``$\sigma$ is surjective'' in Theorem~\ref{thm130316b}; see~(\ref{disc131027a}.1). Also, we show that an affirmative answer to the right inequality in Foxby's Question provides a much shorter proof for Theorem~\ref{thm130316c}, without the assumption of finiteness of $\varphi$; see~(\ref{disc131027a}.2).

\begin{disc}\label{disc131027a}
~(\ref{disc131027a}.1) Using the notation from Construction~\ref{fact130607c}, there exists a diagram $R''\xra{\vf''} S'\xra{\sigma'}T$ of surjective local ring homomorphisms. Since $\sigma\varphi$ is complete intersection at $\fr$, from~\cite[(1.7), (1.8)]{Av} we get that $\sigma'\vf''$ is also complete intersection at $\fr$. Now a similar argument as in the proof of Theorem~\ref{thm130316b}, yields $\cidim_{S'} T<\infty$. Then the left inequality of Foxby's Question implies that $\cfd_{S} T<\infty$, as desired.

~(\ref{disc131027a}.2) Let $R\xra{\dot{\vf}} R'\xra{\vf'} S$ be a Cohen factorization of $\vf$. From the right inequality in Foxby's Question, we get $\cidim_{R'} S<\infty$. Since $\HH_1(K.(\varphi))$ is a free $S$-module, from Lemma~\ref{para131027a} and ~\cite[Proposition 12]{S} we conclude that
$\DD_n(S\mid R',\ell)=0$ for all $n\geq 3$. Now, the assertion follows from~\cite[Lemma (1.7)]{Av} and~\cite[Lemma 8.7]{I}.
\end{disc}

The next corollary of Soto~\cite[Proposition 12]{S} and of Avramov,
Henriques, and \c{S}ega~\cite[(2.5)(3)]{AHS} follows from Theorem~\ref{thm130316c} and Lemma~\ref{para131027a}.

\begin{cor}\label{cor131027a}
Assume that $\varphi$ is surjective, and let $H_1(E)$ be the first homology of the Koszul complex $E$ associated to an arbitrary set of generators for $J:=\Ker \varphi$. If $\cdim_{R} S<\infty$ and $H_1(E)$ is a free $S$-module, then $\aq_{R} S\leq 2$.
\end{cor}

Using Theorem~\ref{thm130316b}, Lemma~\ref{prop130720a} and Corollary~\ref{cor131027a} we have the next result due to Soto~\cite[Proposition 23 (i)$\Longleftrightarrow$(ii)]{S}. Recall that $S$ is called an \emph{algebra retract of
$R$} if there exists a local ring homomorphism $\psi\colon S\to R$ such that
$\varphi\psi=\id_S$.

\begin{cor}\label{cor130316b}
Let $S$ be an algebra retract of $R$ and let $\HH_1(E)$ be the first homology of the Koszul complex $E$ associated to an arbitrary set of generators for $J:=\Ker \varphi$. Then $\aq_R S\leq 2$ if and only if $\cidim_R S<\infty$ and $\HH_1(E)$ is a free $S$-module.
\end{cor}

\begin{fact}\label{fact130528a}
Let $S$ be an algebra retract of $R$ and $\cx_{R} S<\infty$. If $\ch(k)=0$ then $S$ contains the rational numbers. On the other hand $\varphi\colon R\to S$ is a large homomorphism. Let $R\langle U\rangle$ be the acyclic closure
of $\vf$, which is a minimal free resolution for $S$ as an $R$-module by~\cite[Corollary (2.7)]{AI3}. Since $\cx_R S<\infty$, it follows from~\cite[Lemma~ 16]{S} that $\DD_{2n}(S\mid R, -)=0$ for all $n\gg 0$. Now from~\cite[Theorem (7.5)(iv)$\Longrightarrow$(ii)]{AI2} we have $\aq_R S\leq 2$. Thus by Corollary~\ref{cor130316b}, $\cidim_R S<\infty$.
\end{fact}

The above fact answers Soto's Problem 24 in~\cite{S} when $\ch(k)=0$. As an application of this fact we also have the following result.

\begin{cor}
Let $A:=S\otimes_{R}S$, and assume that $\varphi$ is a local flat homomorphism essentially of finite type. Let $\omega : A\to S$ be the natural surjection and $\q= \omega^{-1}(\fn)$. If $\ch(\ell)=0$ and $\cx_{A_{\q}} S<\infty$, then $\varphi$ is locally complete intersection.
\end{cor}

\begin{proof}
Consider the algebra retract $S\rightarrow A_{q}\rightarrow S$. Fact~\ref{fact130528a} implies that
$\aq_{A_{\q}}S\\\leq 2$. Using Jacobi-Zariski exact sequence
arising from $A\rightarrow A_{\q}\rightarrow S$ and the localization property, we have $\aq_A S\leq 2$. Now Jacobi-Zariski exact sequence arising from the
natural diagram $S\rightarrow A\rightarrow S$, and the flat base change imply that $\DD_{n}(S\mid R,-) = 0$, for
$n\geq 2$. Thus $\varphi$ is locally complete intersection by~\cite[Theorem (1.2)]{Av}.
\end{proof}

\begin{disc}\label{disc130711a}
At this time, we do not know if the assumption $\ch(k)=0$
is necessary in Fact~\ref{fact130528a}. However, as we will see in Proposition~\ref{cor130528b}, when we work with $R$ as an algebra retract of its trivial extension by an $R$-module $M$ with $\ann_R(M)=0$, we are able to remove this assumption. This can be done using a result of Herzog.
\end{disc}

Recall that the \emph{trivial extension of $R$  by an $R$-module $M$}, denoted $R\ltimes M$, is the underlying  $R$-module  $R\oplus M$ equipped with a ring structure given by the multiplication rule
$(r,m)\cdot (r',m')=(rr',rm'+r'm)$.
There are ring homomorphisms $\rho\colon R\to R\ltimes M$ with $\rho(r)=(r,0)$ and $\pi\colon R\ltimes M\to R$ with $\pi(r,m)=r$.
Note that the composition $\pi\rho$ is the identity on $R$, that is, $R$ is an algebra retract of $R\ltimes M$.

\begin{lem}\label{lem130711a}
Let $M$ be a finitely generated $R$-module. Then for each non-negative integer $i$ we have $(\mu_R(M))^i\leq\beta^{R\ltimes M}_i(R)$.
\end{lem}

\begin{proof}
By~\cite[Corollary 1]{He} we have $P^{R\ltimes M}_R(t)=1+tP_R^{R\ltimes M}(t)P^R_M(t)$. Hence, by induction on $i$, there exists a sequence $\{\gamma_i\}$ of natural numbers such that $\beta^{R\ltimes M}_i(R)=\gamma_i+(\mu_R(M))^i$. Therefore, $(\mu_R(M))^i\leq\beta^{R\ltimes M}_i(R)$.
\end{proof}

The following result gives a characterization of local
rings admitting non-trivial semidualizing modules. Recall that a finitely generated $R$-module $C$ is called \emph{semidualizing} if $R\cong \Hom[R]CC$ and $\Ext[R]{i}CC=0$ for all $i>0$. For instance, every finitely generated projective $R$-module of rank 1 is
semidualizing. Also, every dualizing module is semidualizing with finite injective dimension.

\begin{prop}\label{cor130528b}
Let $M$ be a finitely generated $R$-module with $\ann_R(M)=0$. Then the following conditions are equivalent:
\begin{itemize}
\item[(i)]
$M\cong R$.
\item[(ii)]
$\aq_{R\ltimes M}R=2$.
\item[(iii)]
$\cidim_{R\ltimes M}R<\infty$.
\item[(iv)]
$\cx_{R\ltimes M}R<\infty$.
\end{itemize}
If one of the above conditions holds, then $\cidim_{R\ltimes M}R=0$ and $\cx_{R\ltimes M}R=1$.
\end{prop}

\begin{proof}
(i)$\implies$(ii)
Assume that $M\cong R$.
It is straightforward to see that $\Ker \pi=(x)$, where $x$
is an exact zero divisor in the sense of~\cite{HS}, that is, $\ann_{R\ltimes M}(x)\cong (R\ltimes M)/x(R\ltimes M)$. Hence,
$\aq_{R\ltimes M} R=2$ by~\cite[1.7, 3.1]{AHS}.

(ii)$\implies$(iii) follows from Theorem~\ref{thm130316b}.

(iii)$\implies$(iv) follows from~\cite[Theorem (5.3)]{AGP1}.

(iv)$\implies$(i)
Assume that $\cx_{R\ltimes M} R<\infty$. Hence, the sequence of the Betti numbers of $R$ over $R\ltimes M$ has polynomial growth. Since, by Lemma~\ref{lem130711a}, for each non-negative integer $i$ we have $(\mu_R(M))^i\leq\beta^{R\ltimes M}_i(R)$,
we conclude that $\mu_R(M)=1$, which implies that $M\cong R$.

Note that when $\cidim_{R\ltimes M}R<\infty$, then $\cidim_{R\ltimes M}R=0$ by~\cite[Lemma 3.14]{SSW}. Note also that when $M\cong R$, $\ker \pi$ is generated by an exact zero divisor $x$. Therefore, $$\cdots\xra{x} R\ltimes R\xra{x}R\ltimes R\xra{\pi} R\to 0$$
is a free resolution of $R$ over $R\ltimes R$, that is, $\cx_{R\ltimes R}R=1$.
\end{proof}

The following is an application of Proposition~\ref{cor130528b} in a concrete case.

\begin{ex}\label{ex131117}
Let $S = R[\underline{\textbf{x}}]/J$, where $\underline{\textbf{x}} = x_1,\cdots,x_n$
for some integer $n\geq 1$, and $J = \langle x_i^{n_i}x_j^{n_j}\mid n_i+n_j=2,\ \text{and}\ 1\leq i\leq j\leq n\rangle$. Then one of
the integers $\cx_S R$, $\cidim_S R$, or $\aq_S R$ is finite if and only if $n=1$.
In this case we have $\cx_S R = 1$, $\cidim_S R=0$, and $\aq_S R=2$.
\end{ex}

\section*{Acknowledgment}

\vspace{0.05in}\noindent The author is deeply grateful to Javier Majadas Soto for his comments on Fact~\ref{fact2}, and he is also grateful to Saeed Nasseh for his comments on this paper.

\providecommand{\bysame}{\leavevmode\hbox to3em{\hrulefill}\thinspace}
\providecommand{\MR}{\relax\ifhmode\unskip\space\fi MR }
% \MRhref is called by the amsart/book/proc definition of \MR.
\providecommand{\MRhref}[2]{%
  \href{http://www.ams.org/mathscinet-getitem?mr=#1}{#2}
}
\providecommand{\href}[2]{#2}

\end{document}